\documentclass[12pt,oneside]{amsart}
\usepackage{amsmath}
\usepackage{amsthm}
\usepackage{amsfonts}
\usepackage{amssymb}
\usepackage{enumerate}
\usepackage{color}

\newtheorem{theorem}{Theorem}

\newtheorem{proposition}[theorem]{Proposition}
\newtheorem{corollary}[theorem]{Corollary}

\theoremstyle{definition}

\theoremstyle{remark}
\newtheorem{remark}[theorem]{Remark}
\DeclareMathOperator{\imm}{Im} 
 \DeclareMathOperator{\re}{Re}

\DeclareMathOperator{\diam}{diam}

\renewcommand{\phi}{\varphi}


\begin{document}

\title{Peak functions in $\mathbb C$-convex domains}

\address{Carl von Ossietzky Universit\"at Oldenburg, Institut f\"ur Mathematik, Postfach 2503, ¨
D-26111 Oldenburg, Germany}

\address{Institute of Mathematics, Faculty of Mathematics and Computer Science, Jagiellonian
University,  \L ojasiewicza 6, 30-348 Krak\'ow, Poland}

\author{Peter Pflug}
\email{Peter.Pflug@uni-oldenburg.de}
\author{W\l odzimierz Zwonek}
\email{wlodzimierz.zwonek@uj.edu.pl}
\thanks{The research was initiated while second author's research stay at the Carl von Ossietzky University of Oldenburg supported by the Alexander von Humboldt Foundation. The second Author was also supported by the OPUS grant no. 2015/17/B/ST1/00996 of the National Science Centre, Poland.}
\today


\begin{abstract} In the paper we show the existence of different types of peak functions in classes of $\mathbb C$-convex domains. As one of tools used in this context is a result on preserving the regularity of $\mathbb C$-convex domains under projection.

\end{abstract}
\maketitle

\section{Introduction.}
A domain $D\subset\mathbb C^n$ is called {\it $\mathbb C$-convex} if for any affine complex line $l$ passing through the domain $D$ the intersection $D\cap l$ is connected and simply connected. The notion of $\mathbb C$-convexity may be seen as a complex analogue of the notion of convexity. However, as it is well known and as we shall also see below the geometry and regularity of different natural phenomena is in the case of $\mathbb C$-convex domains much more complicated than for the convex ones.

It is trivial that the $\mathbb C$-convexity is preserved under complex affine isometries. On the other hand the fact that the image under projection of a bounded $\mathbb C$-convex  domain remains $\mathbb C$-convex is much harder to prove (see e. g. Theorem 2.3.6 in \cite{And-Pas-Sig 2004})  -- recall that the analogous property in the case of convex domains is trivial. This invariance of $\mathbb C$-convexity will be crucial in the subsequent considerations. Another fact that will be used quite often by us is the fact that $\mathbb C$-convex domains are necessarily {\it linearly convex} -- i. e. for any point $p\in\mathbb C^n\setminus D$ there is a complex affine hyperplane passing through $p$, which is disjoint from $D$. A good reference for properties of $\mathbb C$-convex domains is \cite{Hor 1994} or \cite{And-Pas-Sig 2004}.

In this paper we try to see to what extent some of the properties that are satisfied by convex domains may be transferred to $\mathbb C$-convex domains. The problems discussed involve the following general ones that are quite well understood in the case of convex domains.

\begin{itemize}
\item To what extent is the smoothness of the domain preserved under projections?
\item Find general assumptions on a domain sufficient for the existence of different types of peak functions.
\end{itemize}
In Section \ref{section:extreme} we introduce two classes of domains (extreme and strictly) $\mathbb C$-convex. When restricted to $\mathbb C$-convex domains the extreme $\mathbb C$-convex domains admit (global) holomorphic peak functions (Proposition~\ref{proposition:extreme}). Trying to get the peak points in boundaries of $\mathbb C$-convex domains we follow the ideas from the convex situation, which makes us recall the results on the regularity of shadows of convex domains (Section \ref{section:shadow}). Thus  we are led to results on the regularity of shadows of sufficiently regular $\mathbb C$-convex domains (strictly $\mathbb C$-convex and normal) and as a simple corollary we get the proof of the existence of $A(D)$-peak functions in a wide variety of $\mathbb C$-convex domains (Theorem~\ref{theorem:shadow-c-convex} and Corollary~\ref{corollary:peak-function-c-convex}). It is worth noting that both notions of strict $\mathbb C$-convexity and normality seem quite natural. The first one is  an analog of strict convexity and the other one  a generalization of a regularity of a plane domain which appeared in the theory of regularity of the Riemann mapping. In Section \ref{section:strongly-linearly-convex} we show invariance of the strong linear convexity under projection which is in contrast with the situation of the strong pseudoconvexity.

\section{Preliminary results on the existence of peak functions -- case of extreme $\mathbb C$-convex domains.}\label{section:extreme}
Recall that in the paper \cite{Kos-Zwo 2013} a notion of a weak peak function (and a weak peak point) was introduced. For a domain $D\subset\mathbb C^n$, $p\in\partial D$, the holomorphic function $f:D\to\mathbb D$ is called {\it a weak peak function} if $f$ extends continuously to $D\cup\{p\}$ with $f(p)=1$; in this situation $p$ is called  {\it a weak peak point}.
It was proven in  Theorem 2.2 in \cite{Kos-Zwo 2013} that any boundary point of a bounded $\mathbb C$-convex domain admits a weak peak function.
It turns out that under some additional assumption on the domain the peak function may be chosen to be a little more regular. However, at this point we do not yet prove the existence of peak functions
from the algebra $A(D)$ - the space of all continuous function on $\overline{D}$ that are holomorphic on $D$. We start with a discussion on different possibilities of defining a geometric notion of some kind of strict (or extreme) $\mathbb C$-convexity.

\subsection{Extreme (strict) $\mathbb C$-convexity -- a general discussion.} Below we discuss briefly different approaches to a phenomenon of complex geometric extremality (extreme or strict). The different notions, as we shall see, allow different classes of peak functions.

Let $D$ be a bounded domain in $\mathbb C^n$. We call a point $p\in\partial D$ {\it extreme $\mathbb C$-convex} if there are a complex affine hyperplane $H$ and a neighborhood $U$ of $p$ such that $H\cap \overline{D}\cap U=\{p\}$. In case $U$ may be chosen to be $\mathbb C^n$ the point $p$ is called {\it globally extreme $\mathbb C$-convex}. The domain $D$ is called {\it (globally) extreme $\mathbb C$-convex} if any its boundary point is (globally) extreme $\mathbb C$-convex.

Note that all planar domains are extreme $\mathbb C$-convex.

In \cite{Zim 2017} the Author introduced in the class of convex domains the notion of $\mathbb C$-strict convexity. We adopt the notion to a more general class of $\mathbb C$-convex domains. But let us begin without restriction to $\mathbb C$-convex domains.

Let $D\subset\mathbb C^n$ be a bounded linearly convex domain. We call the boundary point  $p\in\partial D$ {\it $\mathbb C$-strictly $\mathbb C$-convex} or shortly {\it strictly $\mathbb{C}$-convex}  if there is a neighborhood $U$ of $p$ such that $U\cap\overline{D}\cap H=\{p\}$ for any supporting hyperplane $H$ of $D$ at $p$ (i. e. $H$ is a complex affine hyperplane with $p\in H$, $H\cap D=\emptyset$). In the case the neighborhood $U$ is chosen to be the whole $\mathbb C^n$ the point $p$ is called {\it globally strictly $\mathbb C$-convex}. The domain $D$ is called {\it (globally) strictly $\mathbb C$-convex} if any of its boundary point is (globally) strictly $\mathbb C$-convex.

Certainly, if $p$ is strictly $\mathbb C$-convex, then it is extreme $\mathbb C$-convex.  In the class of $C^1$-smooth $\mathbb C$-convex domains both notions are equivalent.

The (purely geometric) notion of strict $\mathbb C$-convexity is an intermediate one between the general $\mathbb C$-convexity and the strong linear convexity (to be discussed later).  In the convex case we have analogous notions: strict convexity (geometric condition) and strong convexity (differential condition).

\begin{remark} We have the following simple property. The proof of the property below needs only the definition of the $\mathbb C$-convexity. Therefore, we omit details of it.

Let $D$ be a bounded $\mathbb C$-convex domain in $\mathbb C^n$, $p\in\partial D$. Then
\begin{itemize}
\item
$p$ is extreme $\mathbb C$-convex iff $p$ is globally extreme $\mathbb C$-convex;
\item
$p$ is strictly $\mathbb C$-convex iff $p$ is globally strictly $\mathbb C$-convex.
\end{itemize}
%\end{lemma}
\end{remark}

\begin{remark} Note that one may consider the following (formally) stronger condition than the strict $\mathbb C$-convexity: for any boundary point $p\in\partial D$ and any complex affine line $l$ passing through $p$ and disjoint from $D$ we have $l\cap\overline{D}=\{p\}$. In the class of $\mathbb C$-convex domains both notions are in fact equivalent.

In fact, assume that for some strictly $\mathbb C$-convex domain there are points $p,q\in\partial D$, $p\neq q$ and a complex line $l$ connecting $p$ and $q$  disjoint from $D$. Composing with affine complex isomorphisms we may assume that
$l=\{0\}^{n-1}\times\mathbb C$, $p=0$. Then one can easily see that the projection of $D$ on $\mathbb C^{n-1}$ gives a $\mathbb C$-convex domain with $0^{\prime}\in\partial\pi(D)$. Let $H^{\prime}$ be a supporting hyperplane at $0^{\prime}$ to $\pi(D)$. Then one may easily see that $H:=H^{\prime}\times\mathbb C$ is a supporting hyperplane to $D$ and $p,q\in H$ -- contradiction with the strict $\mathbb C$-convexity of $D$.

Below we shall make use of the above property of strict $\mathbb C$-convex domains.
\end{remark}

\subsection{Peak functions in extreme $\mathbb C$-convex domains.}
A point $p\in\partial D$, where $D$ is a domain in $\mathbb C^n$, is called a {\it (holomorphic) peak point} if there is a holomorphic function (called {\it (holomorphic) peak function})
$f:D\to\mathbb D$ such that for any open neighborhood $U$ of $p$
\begin{equation}
\sup\{|f(z)|:z\in U\setminus D\}<1=\lim_{z\to p}f(z).
\end{equation}
Then we have the following.
\begin{proposition}\label{proposition:extreme}
Let $D$ be a bounded $\mathbb C$-convex domain in $\mathbb C^n$ and let $p\in\partial D$ be extreme $\mathbb C$-convex. Then $p$ is a holomorphic peak point.
\end{proposition}
\begin{proof}
Let $l$ be a complex line intersecting $D$ and passing through $p$, let $H$ be a hyperplane such that $H\cap\overline{D}=\{p\}$. Let $\pi$ be the projection on $l$ (treated later as the complex plane) in the direction of the hyperplane $H$; in particular, $p=\pi(p)$ and $H=\pi^{-1}(p)$. Then $l\cap D\subset\pi(D)$ and $\pi(D)$ is simply connected (use $\mathbb C$-convexity of $D$ -- see \cite{And-Pas-Sig 2004}). Take a sequence $p_j\in D$ which converges to $p$. Then $l\cap\pi(D)\ni \pi(p_j)$ also converges to $p$ implying that $p\in \overline{\pi(D)}$. To conclude that $p\in\partial\pi(D)$ it is sufficient to note that from the definition of the projection $\pi$ we get that $p=\pi(p)\not\in\pi(D)$.
%we argue as follows: assume that $p\in\pi(D)$. Then there exist a point $q\in H\cap D$ with $\pi(q)=p$ and $q\neq p$. By assumption there exists an $(n-1)%$-dimensional ball $B\subset H$ with center $p$ such that $\overline B\cap H=\{p\}$. In particular we have that $\partial B\cap H=\varnothing$. Then we %may take a small disc $\Delta\subset l$ with center $p$ such that $\Delta\times \partial B=\varnothing$. Taking the complex lines $l_j$ through the points $q%$ and $p_j$ it is seen that for large $j$ there is a closed curve around $p_j$ in $l_j$ which is disjoint to $D$. Hence $l_j\cap D$ is not connected. A %contradiction to the $\mathbb C$-convexity of $D$.}
Let $f:\pi(D)\to\mathbb D$ be a function as in Theorem 2.2 in \cite{Kos-Zwo 2013}  (recall $f(\lambda):=\exp(-1/\log(\diam \pi(D)/(\lambda-\pi(p))))$) - note that it is in particular a holomorphic peak function for $\pi(D)$ at $p$). We put $F:=f\circ \pi$. We claim it is a holomorphic peak function for $D$ at $p$.

In fact, let $U$ be a neighborhood of $p$. Then, because of extreme $\mathbb C$-convexity at $p$, we find a neighborhood $V$ of $\pi(p)$ in $l$ such that
$\pi^{-1}(V\cap\pi(D))\cap D\subset U$. Actually, otherwise there would exist a sequence $(q_j)\subset D$ with $q_j\to q\in\overline{D}$, $q\neq p$, $\pi(q_j)\to\pi(p)=p$, so $\pi(q)=\pi(p)$ and thus $q\in H$ -- contradiction with the equality $H\cap\overline{D}=\{p\}$. Then
\begin{multline}
\sup\{|F(z)|:z\in D\setminus U\}\leq\sup\{|f(\lambda)|:\lambda\in\pi(D)\setminus V\}<1\\
=\lim_{\pi(D)\owns\lambda\to \pi(p)}f(\lambda)=\lim_{z\to p}F(z).
\end{multline}
\end{proof}

\begin{remark} Recall that a domain $D\subset \mathbb C^n$ has the  ($P$) property at $p\in\partial D$ if for any neighborhood $U$ of $p$  (see \cite{Com 1998})
\begin{equation}
\lim_{z\to p}\inf\{g_D(z,w):w\in D\setminus U\}=0.
\end{equation}
Here $g_D(p,\cdot)$ denotes {\it the pluricomplex Green function with the pole at $p\in D$}.

The existence of a holomorphic peak function at $p$ implies that $p$ has the property ($P$) (see \cite{Com 1998}). Therefore, any extreme $\mathbb C$-convex point of a bounded $\mathbb C$-convex domain $D$ has the property ($P$). Consequently, it follows from \cite{Nik 2002} that the localization property of invariant metrics holds. In particular, for any neighborhood $U$ of $p$ we have the following (uniform) convergence
\begin{equation}
\lim_{z\to p}\frac{\kappa_D(z;X)}{\kappa_{U\cap D}(z;X)}=1,\; X\in\mathbb C^n, X\neq 0,
\end{equation}
where $\kappa_D$ denotes {\it the Kobayashi metric}.
\end{remark}

\section{Regularity of shadows -- a general discussion.}\label{section:shadow}
Having proven a result on the existence of peak points one could try to get positive results on the existence of peak functions for the function algebra $A(D)$ for more  regular $\mathbb C$-convex domains.

Recall that in the case of convex domains a general idea of constructing such peak functions often relies upon projecting the domain in the direction of the supporting hyperplane onto the plane and then making use of the fact that projections of convex domains are convex, too and then applying the existence of peak function in convex planar domains. Therefore, it is natural from that point of view to study the regularity of projections of $\mathbb C$-convex domains. Recall that $\mathbb C$-convexity is preserved under such mappings; however, unlike in the convex case, not all $\mathbb C$-convex planar domains have all the points being $A(D)$-peak points (consider $\mathbb D\setminus[0,1)$). Thus it is natural that one has to impose some regularity condition on a $\mathbb C$-convex domain to have the image to be regular enough to admit peak functions.

The (well-understood) problem of the regularity of shadows of convex domains  will be the basis of our subsequent considerations. Let $\pi:\mathbb R^{n+m}\to\mathbb R^n$ be the projection onto the first $n$ coordinates. Let $D\subset\mathbb R^{n+m}$ be convex. The domain $\tilde D:=\pi(D)$ is called {\it shadow}.

We should make several remarks here.

\begin{remark} The problem of preserving the regularity of convex sets under projections has been studied intensively.

The results known include for instance the following:
\begin{itemize}
\item The shadow of a $C^{1,1}$-smooth convex domain is $C^{1,1}$-smooth. This follows easily from the description of $C^{1,1}$-regularity with the aid of the ball condition, a well known folklore that can be found (with proof) for instance in
\cite{Aik-Kil-Sha-Zho 2007}.
\item For three dimensional convex domains ($n=2$, $m=1$) we have the following regularity results for the shadows. The shadow of a $C^1$-smooth convex domain is $C^1$-smooth, the shadow of a $C^{2,1}$-smooth convex domain is $C^2$, and the shadow of a real analytic convex domain is $C^{2,\alpha}$-smooth for some $\alpha>0$. All these results are essentially sharp (see \cite{Kis 1986}).
\item In higher dimensional cases the previous regularities of the shadow of smooth convex domains are essentially weaker than in dimension three (see \cite{Sed 1989}, \cite{Bog 1990}).
\end{itemize}
The above results show that even in the class of convex domains a very small relaxation of regularity properties of the domain leads to a decisive change in the regularity of the shadow. Therefore, it is reasonable to suspect that in the case of $\mathbb C$-convex domains the situation may even be more complicated and difficult. As we shall see in the subsequent section this is actually the case.
\end{remark}

\section{Peak functions in strictly $\mathbb C$-convex domains.}\label{section:shadow-c-convex}

In addition to earlier introduced notions of extreme and strict $\mathbb C$-convexity we introduce another natural geometrical condition.

Let $D$ be a bounded domain in $\mathbb C^n$. We say that the point $p\in\partial D$ is {\it normal} if for any neighborhood $U$ of $p$ there is a neighborhood $V\subset U$ of $p$ such that any two points $w,z\in V\cap D$ may be connected by a curve lying entirely in $U\cap D$.

In the case $n=1$ if $p\in\partial D$ is normal, then it is  "ein normaler Randpunkt'' as it is defined in \cite{Beh-Som 1965} (definition on p. 359). Consequently, if all the boundary points of a bounded $\mathbb C$-convex domain in $\mathbb C$ are normal then the Riemann mapping extends homeomorphically onto the closures and the boundary is a Jordan curve (Theorems 41 and 42 in \cite{Beh-Som 1965}). Therefore, if all boundary points of a simply connected domain are normal, then all the boundary points are peak points for the algebra $A(D)$. Moreover, planar simply connected normal domains, i.e. each of its boundary point is a normal one, are fat.

\bigskip

It turns out that the above properties behave well under projections.

\begin{theorem}\label{theorem:shadow-c-convex}
Let $D$ be a bounded strictly $\mathbb C$-convex domain in $\mathbb C^{n+m}$ that is normal (i.e. any boundary point of $D$ is normal).
Then $\tilde D:=\pi(D)\subset\mathbb C^n$ is a strictly  $\mathbb C$-convex normal domain.
\end{theorem}
\begin{proof} The strict $\mathbb C$-convexity of $\tilde D$ follows directly from the strict $\mathbb C$-convexity of $D$.

To prove the normality of $\tilde D$ let $p\in\partial\tilde D$ and let $\tilde U$ be a neighborhood of $p$ and $(p,q)\in\partial D$. Choose a neighborhood $U$ of $(p,q)$ such that $\pi(U)\subset \tilde U$. The normality of $D$ allows us to choose a neighborhood $V_1\subset U$ of $(p,q)$ such that any two points from $V_1\cap D$ may be connected by a curve lying entirely in $U\cap D$. The strict $\mathbb C$-convexity of $D$ allows us to find a neighborhood $V\subset V_1$ of $(p,q)$ such that $\pi^{-1}(\pi(V)\cap\tilde D)\cap D\subset V_1$. Otherwise we could find a sequence $(p_j,q_j)_j\subset D$ such that $p_j\to p$, $q_j\to \tilde q\neq q$, which would imply that the line connecting $(p,q)$ and $(p,\tilde q)$, which is disjoint from $D$, would contain two different points -- a contradiction.

Put $\tilde V:=\pi(V)$. Then one may easily see that for any two points $\tilde w,\tilde z\in\tilde V\cap\tilde D=\pi(V)\cap\pi(D)$ we find $w,z\in D\cap V_1$ such that $\tilde w=\pi(w)$, $\tilde z=\pi(z)$. But $w$ and $z$ may be connected by a curve $\gamma$ lying entirely in $U\cap D$, which implies that the curve $\pi\circ\gamma$ joining $\tilde w$ and $\tilde z$ lies entirely in $\tilde U\cap \tilde D$.

\end{proof}

\begin{remark} Note that bounded $C^1$-smooth strictly $\mathbb C$-convex domains are normal.
\end{remark}

\begin{remark} Another, even stronger notion of normality may also be considered. Namely, the point $p\in\partial D$, where $D$ is a bounded domain, is called {\it strongly normal} if for any neighborhood $U$ of $p$ there is a neighborhood $V\subset U$ of $p$ such that the set $V\cap D$ is connected. Then, as we shall see below,  the above theorem would remain true if the notion 'normal' would be replaced by 'strongly normal'. Since we are mainly interested in the fact that (strong) normal points satisfy the notion of a 'ein normaler Randpunkt' we present the proof of that (formally different) fact below only as an observation.

Assume that the assumption in the above theorem of the normality of $D$ is replaced by the strong normality. To prove the strong normality of $\tilde D$, let $p\in\partial\tilde D$ and let $\tilde U$ be a neighborhood of $p$ and
$(p,q)\in\partial D$. Choose a neighborhood $U$ of $(p,q)$ such that $\pi(U)\subset \tilde U$. The strong normality of $D$ allows us to choose $V\subset U$ such that $V\cap D$ is connected. We claim that there is an $\epsilon>0$ such that $B(p,\epsilon)\cap\tilde D\subset \pi(V\cap D)$ which would finish the proof since $\tilde V:=\pi(V\cap D)\cup B(p,\epsilon)$ would be the desired neighborhood of $p$ (note that $\tilde V\cap\tilde D=\pi(V\cap D)$ is connected).
To prove the claim suppose the opposite. Then there is a sequence $(z^k)\subset\tilde D$ tending to $p$ such that for some $w^k$ we have $(z^k,w^k)\in D\setminus V$. Without loss of generality
$w^k\to w\neq q$. But then $(p,w)\in\partial D$ and the complex affine line $l\subset\{p\}^n\times\mathbb C^m$ that connects $(p,q)$ and $(p,w)$ disjoint from $D$ contains two points $(p,q)$ and $(p,w)$ from $\overline{D}$, which contradicts the strict $\mathbb C$-convexity of $D$.
\end{remark}
Recall that for $p\in\partial D$ the function $f\in A(D)$ is called {\it $A(D)$-peak function} if $f(p)=1$ and $|f(z)|<1$, $z\in\overline{D}\setminus\{p\}$; the point $p$ is then {\it an $A(D)$-peak point}.

\begin{corollary}\label{corollary:peak-function-c-convex} Let $D$ be a bounded strictly $\mathbb C$-convex domain in $\mathbb C^n$. Assume that $D$ is normal (for instance $D$ is $C^1$-smooth). Then any boundary point of $D$ is an $A(D)$-peak point.
\end{corollary}

\begin{remark} It would be desirable to try to extend the previous results to a wider class of domains by relaxing some of the assumptions. However, one should be aware of some limits of possible extensions. Namely, the example of a Hartogs domain in $\mathbb C^2$ from \cite{Jac 2008} (§ 3.2.4, pp. 55--58) shows that the fat $\mathbb C$-convex domain may have some of its boundary not being $A(D)$-peak points and the image under projection may be non-fat. Certainly our results above give also the fatness of the image of a projection of the strictly $\mathbb C$-convex normal domain. That kind of property is another one projections of $\mathbb C$-convex domains can have (see the results on spiral connectedness in e. g. Corollary 2.6.7 in \cite{And-Pas-Sig 2004} or \cite{Zna-Zna 1996}).
\end{remark}

\begin{remark} At this point it should be recalled that M. Range introduced in \cite{Ran 1978} the notion of {\it a totally pseudoconvex domain} which may be seen as a generalization of strict $\mathbb C$-convexity in which the supporting hyperplane is replaced by a kind of a supporting hypersurface. And M. Range showed the existence of local peak points in such domains with $C^1$-boundary. In the case the domain is $C^{\infty}$-smooth the peak functions are in fact global ones. Our results therefore extend partially Range's results with much simpler methods.
\end{remark}

\section{Projection of strongly linearly convex domains.}\label{section:strongly-linearly-convex} We now move to the study of the regularity of the shadow in a much more regular class of $\mathbb C$-convex domains.
Let $D\subset\mathbb C^{n+m}$ be a strongly linearly convex $C^2$-domain. In other words there is a $C^2$-defining function $\rho$
defined on a neighborhood $U$ of $\partial D$, which, in particular, satisfies the following properties $\partial D=\{z\in U:\rho(z)=0\}$, $\rho^{\prime}\neq 0$ on $U$, and $\mathcal H \rho(z)(X)\geq c||X||^2$ for all $X\in T_z^{\mathbb C}\partial D\subset T_z\partial D,\; z\in\partial D$ for some $c>0$. Recall that any strongly linearly convex domain is strongly pseudoconvex and moreover strictly $\mathbb C$-convex. The other implication does not hold, even in the class of $C^2$ smooth $\mathbb C$-convex domains.

Let $\pi$ be the projection onto $\mathbb C^n$. Similarly as earlier we may consider the $\mathbb C$-convex domain (called \textit{shadow}) $\tilde D:=\pi(D)\subset\mathbb C^n$. Then we may define the homeomorphism onto its image $\Phi:\partial\tilde D\to\partial D$ such that $\pi\circ\Phi$ is the identity. In fact note that over any point $z\in\partial\tilde D$ there lies exactly one point $\Phi(z)\in\partial D$. Moreover, it easily seen that $\Phi$ is continuous.

\begin{theorem} Let $D$ be a $C^k$-smooth strongly linearly convex domain in $\mathbb C^{n+m}$, $k\geq 2$. Then $\tilde D$ is a $C^k$-smooth  strongly linearly convex domain. Moreover, $T_{z}^{\mathbb C}\partial\tilde D\times\mathbb C^m=T_{\Phi(z)}^{\mathbb C}\partial D$, $z\in\partial\tilde D$, and $\Phi$ is a $C^k$-diffeomorphism onto the image.
\end{theorem}
\begin{proof} Denote the variables by $(z,w)=(z_1,\ldots,z_n,w_1,\ldots,w_m)$.

Note that $S:=\Phi(\partial\tilde D)=\left\{(z,w)\in U:\frac{\partial\rho}{\partial w}(z,w)=0,\rho(z,w)=0\right\}\subset\partial D$.
Additionally, note that for $(z,w)\in S$ we have
$T_{(z,w)}^{\mathbb C}\partial D\supset\{0\}^{n}\times\mathbb C^m$.

Consider the equation (actually the formula below represents $2m$ real equations)
\begin{equation}
\frac{\partial\rho}{\partial w}=0.
\end{equation}
We shall below solve the above equation near the points $(z,w)\in S$.

Let us calculate the derivative of the above equation with respect to (real) variables: $\re w_1,\imm w_1,\ldots,\re w_m, \imm w_m$.
Note that the Jacobi matrix of the system of equations with respect to the variables is (part of) the Hesse matrix of $\rho$. This part consists of derivatives with respect to the same systems of $2m$ variables. Denote that square $(2m)\times (2m)$ matrix by $\tilde{\mathcal H}$. Note that for all $Y\in\mathbb C^m$ we have $(0,Y)\in T_{(z,w)}^{\mathbb C}\partial D$. Thus we get the following inequality
\begin{equation}
c||Y||^2\leq \mathcal H\rho(z,w)(0,Y)=\tilde{\mathcal H}\rho(z,w)(0,Y)
\end{equation}
which implies that $\det\tilde{\mathcal H}\rho(z,w)>0$. In other words we may apply the implicit mapping theorem to get near $(z,w)\in S$ the solution of the equation $\frac{\partial\rho}{\partial w}=0$ as a $C^{k-1}$ mapping $\varphi:U(z)\to U(w)$. In particular, $\frac{\partial\rho}{\partial w}(z,\varphi(z))=0$.

Now we arrive at the final part of the proof. Define locally $\tilde\rho(z):=\rho(z,\varphi(z))$. Then $\tilde\rho$ is $C^{k-1}$ smooth. Note that for $z\in\partial D$ we have
$\frac{\partial \tilde\rho}{\partial z}(z)=\frac{\partial\rho}{\partial z}(z,\varphi(z))$ which does not vanish (remember that $\frac{\partial\rho}{\partial w}(z,\varphi(z))=0$, $z\in\partial\tilde D$). Moreover, the last formula implies that $\frac{\partial\tilde\rho}{\partial z}$ is $C^{k-1}$ which implies that $\tilde\rho$ is $C^k$. 

The Hesse matrix of $\rho$ on the vectors $X\in T_z^{\mathbb C}\partial D$ satisfies (remember that $(X,0)\in T_{(z,\varphi(z))}^{\mathbb C}\partial D$)
\begin{equation}
\mathcal H\tilde\rho(z)(X)=\mathcal H\rho(z,\varphi(z))(X,0)\geq c||X||^2
\end{equation}
from which we get the positive definiteness of the Hesse matrix on the complex tangent hyperplane.

We should also see whether $\tilde\rho$ satisfies the properties: $\tilde\rho<0$ on $\tilde D$ and $\tilde\rho>0$ on $\mathbb C^n\setminus\overline{\tilde D}$. But shrinking if necessary the neighborhood of the point where we define $\tilde\rho$ we may assume that the infimum of $\rho$ when restricted to $\{z\}\times\mathbb C^m$ is attained in the neighborhood of the point from $\partial D\subset\mathbb C^{n+m}$ where the implicit function theorem is solved.
But at the point where the  infimum is attained the equality $\frac{\partial\rho}{\partial w}=0$ is satisfied (this is a consequence of $\mathbb C$-strict convexity). This implies the desired property and finishes the proof of the strong linear convexity of $\tilde D$.
\end{proof}

\begin{remark} We should point out the following phenomenon that differs the situation of the $\mathbb C$-convex case from the case of pseudoconvex domains. Recall that the projections of pseudoconvex domains (even strongly pseudoconvex ones) need not be pseudoconvex (see \cite{Pfl 1978}). The above result shows that strongly linearly convex domains behave much more like the convex domains rather than the pseudoconvex ones. Recall also that the idea of the above proof is actually the same as the one in the analogous result in the convex case (see Lemma 7.1.5 in \cite{Hor 1994}). In case of a (real) strongly convex domain the same method also applies and gives the invariance of the corresponding notion under projections.
\end{remark}


\begin{thebibliography}{10}

\bibitem{Aik-Kil-Sha-Zho 2007} \textsc{H. Aikawa, T. Kilpel\"ainen, N. Shanmugalingam, X. Zhong}, \textit{Boundary Harnack Principle for $p$-harmonic Functions in Smooth Euclidean Domains}, Potential Analysis, 26(3), 2007, 281--301.

\bibitem{And-Pas-Sig 2004} \textsc{M. Andersson, M. Passare, R. Sigurdsson}, \textit{Complex convexity and analytic
functionals}, Birkhauser, Basel-Boston-Berlin, 2004.

\bibitem{Beh-Som 1965} \textsc{H. Behnke, F. Sommer}, \textit{Theorie der Analytischen Funktionen einer Komplexen Ver\"anderlichen}, Grundlehren der mathematischen Wissenschaften, Springer Verlag Berlin Heidelberg, 1965.

\bibitem{Bog 1990} \textsc{I. A. Bogaevsky}, \textit{Degree of smoothness for visible contours of convex hypersurfaces}, Theory of Singu-
larities and its Applications. Editor: V.I. Arnold. Providence, RI: Amer. Math. Soc., 1990, 119--127.

\bibitem{Com 1998} \textsc{D. Coman}, \textit{Boundary behavior of the pluricomplex Green function}, Ark. Mat. 36, 341--353 (1998).

\bibitem{Hor 1994} \textsc{L. H\"ormander}, \textit{Notions of Convexity}, Birkh\"auser Boston, 1994.

\bibitem{Jac 2008} \textsc{D. Jacquet}, \textit{On complex convexity}, PhD Thesis, University of Stockholm (2008).

\bibitem{Kis 1986} \textsc{Ch. Kiselman}, \textit{How Smooth is the Shadow of a Smooth Convex Body?},
Journal of the London Mathematical Society, vol. s2-33, Issue 1, 1986,  101--109.

\bibitem{Kos-Zwo 2013} \textsc{\L. Kosi\'nski, W. Zwonek}, \textit{Proper holomorphic mappings vs. peak points and Shilov boundary}, Ann. Polon. Math. 107 (2013), no. 1, 97--108.

\bibitem{Nik 2002} \textsc{N. Nikolov}, \textit{Localization of invariant metrics}, Arch. Math. 79 (2002) 67--73.

\bibitem{Pfl 1978} \textsc{P. Pflug}, \textit{Ein $C^{\infty}$-glattes streng pseudokonvexes Gebiet in $\mathbb C^3$ mit nicht holomorph-konvexer Projektion}, Abh. Math. Sem. Univ. Hamburg 47 (1978), 92--94.

\bibitem{Ran 1978} \textsc{M. R. Range}, \textit{The Carath\'odory metric and holomorphic maps on a class of weakly pseudoconvex domains}, Pacific J. Math., 78(1), (1978), 173--189.

\bibitem{Sed 1989} \textsc{V.D. Sedykh}, \textit{An infinitely smooth compact convex hypersurface with a shadow whose boundary is
not twice--differentiable}, Funct. Anal. Appl., 1989, 23(3), 246--249.

\bibitem{Zim 2017} \textsc{A. M. Zimmer}, \textit{Characterizing domains by the limit set of their automorphism group},
Advances in Mathematics, 308: 438--482, 2017.

\bibitem{Zna-Zna 1996} \textsc{ S. V. Znamenski'i, L.N. Znamenskaya}, \textit{Spiral connectedness of the sections and projections of $\mathbb C$--convex sets}, Math. Notes, 59 (1996), 253--260.
\end{thebibliography}
\end{document}